\documentclass[reqno, 11pt]{amsart}
\usepackage{amsmath,mathrsfs}
\usepackage{amsfonts}
\usepackage{amssymb}
\usepackage{graphicx,color}
\usepackage[square,numbers]{natbib}
\usepackage{verbatim}
\newtheorem{theorem}{Theorem}
\newtheorem{claim}[theorem]{Claim}
\newtheorem{corollary}[theorem]{Corollary}
\newtheorem{lemma}[theorem]{Lemma}
\newtheorem{proposition}[theorem]{Proposition}

\theoremstyle{definition}
\newtheorem{definition}[theorem]{Definition}

\theoremstyle{remark}
\newtheorem{remark}[theorem]{Remark}

\numberwithin{theorem}{section}
\numberwithin{equation}{section}

\def\XXint#1#2#3{{\setbox0=\hbox{$#1{#2#3}{\int}$}
     \vcenter{\hbox{$#2#3$}}\kern-.5\wd0}}

\newcommand{\dd}{\; \mathrm{d}}

\newcommand{\bbG}{\mathbb{G}}
\newcommand{\bbR}{\mathbb{R}}
\newcommand{\bbN}{\mathbb{N}}

\begin{document}
\title[Structure of Porous Sets in Carnot Groups]{Structure of Porous Sets in Carnot Groups}
\author[Andrea Pinamonti]{Andrea Pinamonti}
\address[Andrea Pinamonti]{Department of Mathematics, University of Trento, Via Sommarive 14, 38050 Povo (Trento), Italy}
\email{Andrea.Pinamonti@gmail.com}
\author[Gareth Speight]{Gareth Speight}
\address[Gareth Speight]{Department of Mathematical Sciences, University of Cincinnati, 2815 Commons Way, Cincinnati, OH 45221, United States}
\email{Gareth.Speight@uc.edu}

\date{\today}

\renewcommand{\subjclassname} {\textup{2010} Mathematics Subject Classification}


\begin{abstract}
We show that any Carnot group contains a closed nowhere dense set which has measure zero but is not $\sigma$-porous with respect to the Carnot-Carath\'eodory (CC) distance. In the first Heisenberg group we observe that there exist sets which are porous with respect to the CC distance but not the Euclidean distance and vice-versa. In Carnot groups we then construct a Lipschitz function which is Pansu differentiable at no point of a given $\sigma$-porous set and show preimages of open sets under the horizontal gradient are far from being porous.
\end{abstract}

\maketitle

\section{Introduction}\label{intro}

A Carnot group (Definition \ref{Carnotdef}) is a simply connected Lie group whose Lie algebra admits a stratification. Carnot groups have translations, dilations, Haar measure and points are connected by horizontal curves (Definition \ref{horizontalcurve}), which are used to define the Carnot-Carath\'eodory (CC) distance (Definition \ref{ccdistance}). With so much structure, the study of analysis and geometry in Carnot groups is an active and interesting research area \cite{BLU07, CDPT07, Gro96, Mon02, SC16, Vit14}. 

Many interesting geometric and analytic problems have been studied in the context of Carnot groups. For example, a geometric notion of intrinsic Lipschitz function between subgroups of a general Carnot group was introduced in \cite{FSC2} to study rectifiable sets \cite{FSSC2,Mag} and minimal surfaces \cite{CMPSC1,CMPSC2,SCV}. Moreover, Carnot groups have been applied to study degenerate equations, control theory and potential theory \cite{BLU07}. These problems are highly non-trivial due to the complexity of the geometry of Carnot groups. For instance, any Carnot group (except for Euclidean spaces themselves) contains no subset of positive measure that is bi-Lipschitz equivalent to a subset of a Euclidean space \cite{Sem96}. This follows from Pansu's theorem (Theorem \ref{PansuThm}), a generalization of Rademacher's theorem asserting that Lipschitz maps between Carnot groups are differentiable almost everywhere \cite{Pan89, Mag01}. 

A set in a metric space is (upper) porous (Definition \ref{def_porous}) if each of its points sees nearby relatively large holes in the set on arbitrarily small scales. A set is $\sigma$-porous if it is a countable union of porous sets. Properties and applications of porous sets are surveyed in \cite{Zaj87, Zaj05}. Porous sets have applications to differentiability in the linear setting. For example, they were recently used in \cite{PS15} to show that if $n>1$ then there exists a Lebesgue null set $N\subset \mathbb{R}^{n}$ such that every Lipschitz map $f\colon \bbR^{n}\to \bbR^{n-1}$ is differentiable at a point of $N$. Applications of porosity to differentiability also exist in the non-linear setting of Carnot groups. For instance, \cite{PS16.2} showed that if $\bbG$ is a Carnot group and $f\colon \bbG \to \bbR$ is a Lipschitz map, then there exists a $\sigma$-porous set $A$ such that if $f$ is differentiable at $x\in \bbG\setminus A$ in all horizontal directions then $f$ is Pansu differentiable at $x$. Hence it is also interesting to study porous sets and their applications in Carnot groups.

In Section \ref{porousproperties} we investigate the structure of porous sets themselves in Carnot groups. Every $\sigma$-porous set in a metric space is of first category, which means it is a countable union of nowhere dense sets. In $\bbR^{n}$, $\sigma$-porous sets also have Lebesgue measure zero. However, proving a set is $\sigma$-porous gives a stronger result than proving it is of first category or has measure zero: in $\bbR^{n}$ there exists a closed nowhere dense set of Lebesgue measure zero which is not $\sigma$-porous \cite[Theorem 2.3]{Zaj05}. We show there is a natural analogue of this result in Carnot groups (Theorem \ref{poroussmall}).

Any Carnot group can be represented in coordinates as a Euclidean space $\bbR^n$ equipped with some additional structure. Hence one can compare porosity with respect to the CC and Euclidean distances. We show that, at least in the first Heisenberg group, the two notions differ: for each distance, one can construct a set which is porous with respect to the given distance but not porous with respect to the other distance (Proposition \ref{enotcc} and Proposition \ref{ccnote}). This does not follow immediately from the fact that the two distances are not Lipschitz equivalent: if $(M,d)$ is a metric space then the fractal metric $d^{\varepsilon}$, $0<\varepsilon<1$, need not be Lipschitz equivalent to $d$ but gives the same family of porous sets (Remark \ref{sameporoussets}).

In Section \ref{porousnondifferentiability} we give another connection between porosity and differentiability in Carnot groups. We adapt Euclidean arguments from \cite{HMZW97} to show that for any $\sigma$-porous set $P$ in a Carnot group $\bbG$, one can find a real-valued Lipschitz function on $\bbG$ which is not even Pansu subdifferentiable on $P$ (Proposition \ref{nonsubdiff}). As a consequence, a universal differentiability set in a Carnot group cannot be $\sigma$-porous (Corollary \ref{udsnotporous}). Universal differentiability sets are sets which contain a point of differentiability for every real-valued Lipschitz function. Such sets were investigated in \cite{DM11, DM12, DM14} in Euclidean spaces and in \cite{PS16.1} in the Heisenberg group.

In Section \ref{gradientproblem} we use the above mentioned non-differentiable functions and arguments from \cite{HMZW97} to investigate the horizontal gradient. The horizontal gradient is defined like the ordinary gradient in Euclidean spaces, but using only horizontal directional derivatives. We show that, as in the Euclidean case, the preimage of an open set under the horizontal gradient mapping is either empty or far from being porous (Theorem \ref{gradienttheorem}). In the Euclidean case this result is related to the so-called \emph{gradient problem}. Assume $n\geq 2$, $\Omega\subset \bbR^n$ is an open set and $f: \Omega\to \bbR$ is an everywhere differentiable function. The gradient problem asks whether it is true or not that fixed $\tilde\Omega\subset\bbR^n$ open, the set $\{p\in\Omega\ |\ \nabla f(p)\in \tilde\Omega\}$ is either empty or of positive $n$-dimensional Lebesgue measure. It is known that the answer is affirmative for $n=1$ \cite{Cla, D} and negative in higher dimensions \cite{Buc}. We are not aware of any result of this type in the context of Carnot groups.

\medskip

\noindent \textbf{Acknowledgement.} 
Part of this work was carried out while A. P. was a postdoctoral researcher at the University of Bologna and G. S. was supported by a Taft Summer Research Fellowship from the Charles Phelps Taft Research Center. The project was finalized when A. P. visited the Department of Mathematical Sciences of the University of Cincinnati. A.P. wishes to express his gratitude to the University of Cincinnati for the stimulating atmosphere and the excellent working conditions. 

The authors thank helpful referees for valuable remarks and corrections.

\section{Preliminaries}\label{prelim}
In this section we define Carnot groups and porous sets. We refer the interested reader to \cite{BLU07, SC16} for an introduction to Carnot groups and \cite{Zaj87, Zaj05} for information about porosity.

\subsection{Carnot Groups}
\begin{definition}\label{Carnotdef}
A simply connected finite dimensional Lie group $\bbG$ is a \emph{Carnot group of step $s$} if its Lie algebra $(\mathfrak{g}, [\cdot,\cdot])$ is \emph{stratified of step $s$}, which means that there exist non-zero linear subspaces $V_1, \ldots ,V_s$ of $\mathfrak{g}$ such that
\[\mathfrak{g}=V_1\oplus \cdots \oplus V_s,\]
with
\[[V_1,V_{i}]=V_{i+1} \mbox{ if }1\leq i\leq s-1, \mbox{ and } [V_1,V_s]=\{0\}.\]
Here
\[[V,W]=\mathrm{span}\{[X,Y]: X\in V,\ Y\in W\},\]
where $[X, Y]$ denotes the Lie bracket in the Lie algebra.
\end{definition}

\subsection{Coordinates on Carnot Groups}
The \emph{exponential map} $\exp \colon \mathfrak{g}\to \bbG$ is defined by $\exp(X)=\gamma(1)$, where
$\gamma \colon[0,1]\to \bbG$ is the unique solution to the initial value problem:
\[\gamma'(t)=X(\gamma(t)),\qquad \gamma(0)=0.\]
The exponential map $\exp\colon \mathfrak{g}\to \bbG$ is a diffeomorphism. Throughout the paper we fix a basis $X_{1}, \ldots, X_{n}$ of $\mathfrak{g}$ adapted to the stratification, in which a basis of $V_{i+1}$ follows a basis of $V_{i}$ for each $i$. Define $m=\dim(V_1)$ and notice $X_{1}, \ldots, X_{m}$ is then a basis of $V_1$. Any $x\in \bbG$ can be written uniquely as
\[x=\exp(x_{1}X_{1}+\ldots +x_{n}X_{n})\]
for some $x_{1},\ldots, x_{n}\in \bbR$. We identify $x$ with $(x_{1},\ldots, x_{n})\in \bbR^{n}$ and $\bbG$ with $(\bbR^{n},\cdot)$, where the group operation on $\bbR^{n}$ is determined by the Baker-Campbell-Hausdorff formula on $\mathfrak{g}$. This is known as \emph{exponential coordinates of the first kind}. With this identification the Lebesgue measure $\mathcal{L}^{n}$ on $\bbR^n$ is a \emph{Haar~measure} on $\mathbb{G}$.

For any $\lambda>0$, the \emph{dilation} $\delta_{\lambda}\colon \mathfrak{g}\to \mathfrak{g}$ is defined as the linear map satisfying $\delta_{\lambda}(X)=\lambda^{i}X$ whenever $X\in V_{i}$. Dilations are extended to $\bbG$ so as to satisfy the equality $\exp \circ \delta_{\lambda} = \delta_{\lambda}\circ \exp$. Dilations satisfy the equality $\delta_{\lambda}(xy)=\delta_{\lambda}(x)\delta_{\lambda}(y)$ for $x, y\in \bbG$ and $\lambda>0$.

Define the \emph{projection} $p\colon \bbR^{n} \to \bbR^{m}$ by $p(x_1,\ldots, x_n)=(x_{1}, \ldots, x_{m})$. 

\subsection{Carnot-Carath\'eodory Distance}

\begin{definition}\label{horizontalcurve}
An absolutely continuous curve $\gamma\colon [a,b]\to \bbG$ is \emph{horizontal} if there exist $u_{1}, \ldots, u_{m}\in L^{1}[a,b]$ such that for almost every $t\in [a,b]$:
\[\gamma'(t)=\sum_{i=1}^{m}u_{i}(t)X_{i}(\gamma(t)).\]
Define the horizontal length of such a curve $\gamma$ by $L(\gamma)=\int_{a}^{b}|u(t)|\dd t$, where $u=(u_{1}, \ldots, u_{m})$ and $|\cdot|$ denotes the Euclidean norm on $\bbR^{m}$.
\end{definition}

\begin{definition}\label{ccdistance}
The \emph{Carnot-Carath\'eodory distance (CC distance)} between points $x, y \in \bbG$ is defined by:
\[d_{c}(x,y)=\inf \{ L(\gamma): \, \gamma \colon [0,1]\to \bbG \mbox{ horizontal joining }x\mbox{ to }y \}.\]
\end{definition}

The Chow-Rashevskii Theorem implies that any two points of $\bbG$ can be connected by horizontal curves \cite[Theorem 9.1.3]{BLU07}. It follows that the CC distance is indeed a distance on $\bbG$. The following identities hold for $x, y, z\in \bbG$ and $r>0$: 
\[d_{c}(zx,zy)=d_{c}(x,y), \qquad d_{c}(\delta_{r}(x),\delta_{r}(y))=rd_{c}(x,y).\]
For brevity we write $d_{c}(x)$ instead of $d_{c}(x,0)$. Since $\bbG$ is identified with $\bbR^{n}$, we may compare the CC distance $d_{c}$ with the Euclidean distance $d_{e}$. These distances induce the same topology on $\bbR^{n}$ but are not Lipschitz equivalent.

\subsection{Pansu Differentiability and Directional Derivatives}

\begin{definition}\label{subdiff}
A map $L\colon \bbG \to \bbR$ is \emph{group linear} if the following identities hold for every $x, y\in \bbG$ and $r>0$:
\[L(xy)=L(x)+L(y), \qquad L(\delta_{r}(x))= r L(x).\]

A map $f\colon \bbG \to \bbR$ is \emph{Pansu subdifferentiable} at $x_0\in \bbG$ if there exists a group linear map $L\colon \bbG \to \bbR$ such that
\begin{equation}\label{subdiff2}\liminf_{h\to 0} \frac{f(x_{0}h)-f(x_{0})-L(h)}{d_{c}(h)}\geq  0.\end{equation}

Such a map $f$ is \emph{Pansu differentiable} at $x_0$ if there exists a group linear map $L\colon \bbG\to \bbR$ such that
\[\lim_{h\to 0} \frac{f(x_{0}h)-f(x_{0})-L(h)}{d_{c}(h)}= 0.\]

\end{definition}

If the map $L$ in the definition of Pansu differentiability exists then it is unique, called the \emph{Pansu differential}, and we denote it by $df(x_{0})$. For simplicity we now state the celebrated Pansu theorem for scalar targets, though a similar result holds when the target is any Carnot group \cite{Pan89}.

\begin{theorem}[Pansu's Theorem]\label{PansuThm}
Let $f\colon \bbG \to \bbR$ be a Lipschitz map. Then $f$ is Pansu differentiable almost everywhere.
\end{theorem}

We define horizontal directional derivatives of a Lipschitz function by composition with horizontal lines.

\begin{definition}\label{directionalderivative}
Suppose $f\colon \bbG\to \bbR$, $x\in \bbG$ and $X\in V_{1}$. We say that $f$ is \emph{differentiable at $x$ in direction $X$} if the following limit exists:
\[Xf(x)=\lim_{t\to 0}\frac{f(x\exp(tX))-f(x)}{t}.\]
\end{definition}

If $x\in \bbG$ and $X\in V_{1}$ is horizontal, the map $t\mapsto x\exp(tX)$ is Lipschitz. Hence if $f\colon \bbG\to \bbR$ is Lipschitz, the composition $t\mapsto f(x\exp(tX))$ is a Lipschitz mapping from $\bbR$ to $\bbR$ so is differentiable almost everywhere. Thus Lipschitz functions have directional derivatives along each horizontal line, at almost every point and in the direction of the line. It follows from the definition of $d_{c}$ that if $f\colon \bbG \to \bbR$ is Lipschitz then $|X_{i}f(x)|\leq \mathrm{Lip}(f)$ whenever $1\leq i\leq m$ and $X_{i}f(x)$ exists. Here we recall that $X_{1},\ldots, X_{m}$ is our fixed basis of $V_1$, which we used to define horizontal length and CC distance.

\begin{definition}
Suppose $f\colon \bbG\to \bbR$, $x\in \bbG$ and $X_{i}f(x)$ exists for every $1\leq i\leq m$.  The \emph{horizontal gradient} of $f$ at $x$ is defined by
\[\nabla_{H}f(x)=(X_1f(x), \ldots ,X_{m}f(x)) \in \bbR^{m}.\]
\end{definition}

As observed in \cite[Remark 3.3]{MS01}, if $f$ is Pansu differentiable at $x_0$ then $X_{i}f(x)$ exists for every $1\leq i\leq m$ and
\begin{equation}\label{gradform}df(x_0)(h)=\langle \nabla_{H}f(x_0), p(h) \rangle.\end{equation}
Here $\langle \cdot, \cdot \rangle$ denotes the Euclidean inner product on $\bbR^{m}$.

\subsection{Porous Sets}

When a metric space is clear from the context, we denote the open ball of centre $x$ and radius $r>0$ by $B(x,r)$.

\begin{definition}\label{def_porous}
Let $(M, d)$ be a metric space and $E\subset M$.

Given $\lambda\in (0,1)$ and $a\in E$, we say $E$ is \emph{$\lambda$-porous at $a$} if there is a sequence of points $x_{n}\in M$ with $x_{n}\to a$ such that
\[B(x_n,\lambda d(x_n,a))\cap E=\varnothing \qquad \mbox{ for every }n\in \mathbb{N}.\]
The set $E$ is \emph{porous at $a$} if it is $\lambda$-porous at $a$ for some $\lambda\in (0,1)$.

The set $E$ is \emph{porous} if there is $\lambda \in (0,1)$ such that $E$ is $\lambda$-porous at $a$ for every point $a\in E$, with $\lambda$ independent of the point $a$. 

A set is \emph{$\sigma$-porous} if it is a countable union of porous sets. 
\end{definition}

Porous sets in metric spaces are nowhere dense and $\sigma$-porous sets in metric spaces are of first category, which means they can be written as a countable union of nowhere dense sets. However, in any non-empty topologically complete metric space without isolated points, there exists a closed nowhere dense set which is not $\sigma$-porous \cite[Theorem 2.4]{Zaj05}. 

Unless otherwise stated, porosity in a Carnot group will mean porosity with respect to the CC distance (or a Lipschitz equivalent distance).

\section{Geometry of Porous Sets in Carnot Groups}\label{porousproperties}

In this section we compare CC porosity with Euclidean porosity and other notions of smallness of sets.

\subsection{Measure and Porosity in Carnot Groups}

We begin by observing that porous sets in Carnot groups have measure zero. The simple proof does not depend on the structure of Carnot groups; it is known that porous sets have measure zero in any metric space equipped with a doubling measure.

\begin{proposition}\label{porousnull}
Porous sets in Carnot groups have measure zero.
\end{proposition}

\begin{proof}
Haar measure $\mathcal{L}^{n}$ on $\bbG$ equipped with the CC distance is doubling, so the Lebesgue differentiation theorem holds \cite[Theorem 1.8]{Hei01}. Given $P\subset \bbG$ Lebesgue measurable, applying the Lebesgue differentiation theorem to the characteristic function of $P$ implies that the density
\[\lim_{r\downarrow 0} \frac{\mathcal{L}^{n}(P\cap B(x,r))}{\mathcal{L}^{n}(B(x,r))}\]
is equal to $1$ for almost every $x\in P$ and equal to $0$ for almost every $x\notin P$. 

Recall that a set is of class $G_{\delta}$ if it can be expressed as a countable intersection of open sets. We claim that to prove the theorem it suffices to show that every Lebesgue measurable porous set has Lebesgue measure zero. This follows from the fact that every porous set is contained in a porous set of class $G_{\delta}$. To see this fact, suppose $P$ is porous with porosity constant $c$. Let $H=(\bbG \setminus P)^{\circ}$ and define 
\[Q=(\bbG\setminus H)\cap\bigcap_{n=1}^{\infty}Q_{n},\]
where
\[Q_{n}= \{x \in \bbG: \exists \, z \in H, \, d_{c}(z,x)<\frac{1}{n}, \, \overline{B(z,(c/2)d(z,x))}\subset H\}.\]
The sets $Q_{n}$ are open while $\bbG\setminus H$ is closed (hence $G_{\delta}$) so $Q$ is $G_{\delta}$. It is easy to check that $Q$ is a porous set containing $P$.

Now suppose $P$ is a Lebesgue measurable porous set with $\mathcal{L}^{n}(P)>0$. By the Lebesgue differentiation theorem, there exists $x\in P$ such that
\begin{equation}\label{pnpf}
\lim_{r\downarrow 0} \frac{\mathcal{L}^{n}(P\cap B(x,r))}{\mathcal{L}^{n}(B(x,r))}=1.
\end{equation}
Since $P$ is porous, there are $0<\lambda<1$ and $x_{n}\to x$ such that:
\begin{equation}\label{emptyint}
B(x_{n},\lambda d(x_{n},x))\cap P=\varnothing \mbox{ for every }n\in \mathbb{N}.
\end{equation}
Using \eqref{pnpf} with radii $2d(x_{n},x)\to 0$ and observing
\[B(x_{n},\lambda d(x_{n},x))\subset B(x,2d(x_{n},x)),\]
equation \eqref{emptyint} implies:
\[ \lim_{n\to \infty} \frac{\mathcal{L}^{n}(B(x_{n},\lambda d(x_{n},x)))}{\mathcal{L}^{n}(B(x,2d(x_{n},x)))}=0.\]
This contradicts the fact that $\mathcal{L}^{n}$ is doubling. Hence $\mathcal{L}^{n}(P)=0$, which proves the proposition.
\end{proof}

Proposition \ref{porousnull} implies that the collection of $\sigma$-porous sets is contained in the collection of first category measure zero sets. In $\bbR^{n}$ there exists a closed nowhere dense set of Lebesgue measure zero which is not $\sigma$-porous \cite[Theorem 2.3]{Zaj05}. We now prove a similar statement in Carnot groups.

\begin{theorem}\label{poroussmall}
There exists a closed nowhere dense set in $\bbG$ which has measure zero but is not $\sigma$-porous.
\end{theorem}

Before proving Theorem \ref{poroussmall} we prove the following lemma.

\begin{lemma}\label{puc}
Let $(x,y)\in \mathbb{R}\times \bbR^{n-1}=\bbG$. Then for every $t\in \bbR$, there exists $\tau\in \mathbb{R}^{n-1}$ such that
\[d_{c}((t,\tau),(x,y))=|t-x|.\]
\end{lemma}

\begin{proof}[Proof of Lemma \ref{puc}]
Recall that we are using exponential coordinates of the first kind. By \cite[Proposition 2.2.22]{BLU07} we can write the group law of $\bbG$ as
\begin{align}\label{grouplaw}
(p\cdot q)_i=p_i+q_i+\mathcal{R}_i(p,q)\quad  \mbox{ for every }p,q\in\bbG \mbox{ and } i=1,\ldots, n,
\end{align}
where $\mathcal{R}_i(p,q)$ is a polynomial function depending only on $p_k$ and $q_k$ with $k<i$ and $\mathcal{R}_i\equiv 0$ for all $i\leq m$. Let $y=(y_2,\ldots, y_n)$. We will define $\tau=(\tau_2,\ldots, \tau_n)$ by induction. Let $\tau_2:=y_2$ and for $3\leq i\leq n$ define:
\[
\tau_i:=y_i+\mathcal{R}_{i}((t, \tau)^{-1},(x,y)).
\]
Such a definition is justified because, since $p^{-1}=-p$ for $p\in \bbG$, the formula for $\tau_i$ depends only on $\tau_j$ for $j<i$.
Using \eqref{grouplaw}, it is easy to see that $(t,\tau)^{-1}\cdot (x,y)=(x-t,0,\ldots, 0)$. The conclusion follows by noticing that
\[d_c((t,\tau), (x,y))=d_c(x-t,0,\ldots, 0)=|x-t|.\]
\end{proof}

\begin{proof}[Proof of Theorem \ref{poroussmall}]
Using \cite[Theorem 2.3]{Zaj05}, choose a closed nowhere dense set $N\subset \bbR$ which has Lebesgue measure zero but is not $\sigma$-porous. Then
\[N\times \bbR^{n-1}\subset \bbR^{n}=\bbG\]
is closed and nowhere dense in $\bbG$ (since the distances $d_{c}$ and $d_{e}$ are topologically equivalent) and clearly has Lebesgue measure zero. It remains to show that $N\times \bbR^{n-1}$ is not $\sigma$-porous with respect to $d_{c}$. 
Arguing as in \cite[Lemma 3.4]{Zaj76}, it suffices to check that if $F\subset \bbR$ and $F\times \bbR^{n-1}$ is porous in $\bbG$ then $F$ is porous in $\bbR$. Fix $a\in F$. Since $F\times \bbR^{n-1}$ is porous in $\bbG$ at $(a,0)$, there exists a constant $0<c<1$ (independent of $a$) and sequences $x^{k}\in \bbR$, $y^{k}\in \bbR^{n-1}$, $r^{k}>0$ such that
\begin{enumerate}
	\item $B((x^{k},y^{k}),r^{k})\cap (F\times \bbR^{n-1})=\varnothing$ for all $k\in\mathbb{N}$.
	\item $r^{k}> cd_{c}((a,0),(x^{k},y^{k}))$ for all $k\in\mathbb{N}$.
	\item $x^{k}\to a$ and $y^{k}\to 0$.
\end{enumerate}
The definition of $d_{c}$ implies $d_{c}((a,0),(x^{k},y^{k}))\geq |x^{k}-a|$. Hence (2) implies $r^{k}>c|x^{k}-a|$. If $t\in \bbR$ and $|t-x^{k}|<r^{k}$ then by Lemma \ref{puc}, there exists $\tau\in \bbR^{n-1}$ such that $(t,\tau)\in B((x^{k},y^{k}),r^{k})$. Hence (1) implies $t\notin F$, so $(x^{k}-r^{k},x^{k}+r^{k})\cap F=\varnothing$. Combining these observations with (3) shows that $F$ is porous and concludes the proof.
\end{proof}

\subsection{Comparison of Euclidean and CC Porosity}

\begin{remark}\label{sameporoussets}
If $d_{1}$ and $d_{2}$ are Lipschitz equivalent distances then a set is porous with respect to $d_{1}$ if and only if it is porous with respect to $d_{2}$. Notions of porosity may be the same even if the two distances are not Lipschitz equivalent: it is easy to show that if $(M,d)$ is a metric space then $d$ and any snowflaked metric $d^{\varepsilon}$, $0<\varepsilon<1$, give the same porous sets.
\end{remark}

Recall that $\bbG$ is identified with $\bbR^{n}$ in coordinates, so on $\bbG$ we can consider both the CC distance and the Euclidean distance. We next show that if $\bbG$ is the first Heisenberg group then these two distances give incomparable families of porous sets.

\begin{definition}\label{Heisenberg}
The first Heisenberg group $\mathbb{H}^1$ is $\mathbb{R}^{3}$ equipped with the non-commutative group law:
\begin{equation}\label{Hgplaw}
(x,y,t) (x',y',t') = (x+x', \ y+y', \ t+t'-2( xy' - yx')).
\end{equation}
The Koranyi distance on $\mathbb{H}^{1}$ is defined by:
\begin{equation}\label{Koranyi}
d_{k}(a,b)=\|a^{-1}b\|_{k}, \mbox{ where } \|(x,y,t)\|_{k}=((x^{2}+y^{2})^{2}+t^{2})^{1/4}.
\end{equation}
For brevity we write $d_{k}(a)$ instead of $d_{k}(a,0)$.
\end{definition}

The Koranyi distance is Lipschitz equivalent to the CC distance on $\mathbb{H}^{1}$, so a set is porous with respect to the Koranyi distance if and only if it is porous with respect to the CC distance. When necessary, we use the notation $B_{e}(a,r)$ and $B_{k}(a,r)$ to distinguish between balls in $\mathbb{R}^{3}$ with respect to the Euclidean or Koranyi distance. The following estimate is clear from \eqref{Koranyi}:
\begin{equation}\label{simpleest}
d_{k}(x,y,t)\geq \max \left( |(x,y)|,\sqrt{|t|} \right).
\end{equation}

Let $C$ be the middle third Cantor set. Note $C$ is $(1/3)$-porous as a subset of $[0,1]$ with the Euclidean distance. If $\alpha, \beta \in \bbR$ and $S\subset \bbR$, then we define $\alpha +\beta S=\{\alpha+\beta s \colon s\in S\}$. If $S\subset \bbR$ is porous then $\alpha +\beta S$ is also porous, with the same porosity constant as $S$.

\begin{lemma}\label{intobs}
For $n, k\in \bbN$ with $0\leq k\leq 2^{n-1}$, let $A_{n,k}$ be the translated dilated Cantor set:
\begin{equation}\label{Ank}
A_{n,k}:=(2^{-n} + k2^{-2n}) + 2^{-2n}C.
\end{equation}
Then:
\begin{enumerate}
\item The intersection $A_{n,k}\cap A_{m,l}$ is at most one point if $(n,k)\neq (m,l)$.
\item Each set $A_{n,k}$ is $(1/3)$-porous as a subset of the interval
\[[2^{-n} + k2^{-2n}, 2^{-n} + (k+1)2^{-2n}].\]
\item For every $0<t<1$:
\[[t, t+4t^{2}]\cap \bigcup_{n=1}^{\infty} \bigcup_{k=0}^{2^n - 1}A_{n,k}\neq \varnothing.\]
\end{enumerate}
\end{lemma}

\begin{proof}
The first two assertions are clear from the definition of $A_{n,k}$. To prove the third, let $t \in (0,1)$ then choose $n\in \bbN$ such that $2^{-n}\leq t<2^{-(n-1)}$. Using \eqref{Ank}, we see that the interval $[t,t+2^{-2(n-1)}]$ intersects $A_{n,k}$ or $A_{n+1,k}$ for some $k$. Since
\[2^{-2(n-1)}=4\cdot 2^{-2n}\leq 4t^2,\]
we deduce that $[t,t+4t^{2}]$ intersects $A_{n,k}$ or $A_{n+1,k}$. This proves the lemma.
\end{proof}

\begin{proposition}\label{enotcc}
Define the cone $\Lambda$ by
\[\Lambda:=\{(x,y,t)\in \mathbb{H}^1\ \colon |t| \leq |(x,y)|\}.\]
Define
\[P_{e}:=\Lambda \cap \Big(\{0\}\cup \bigcup_{n=1}^{\infty} \bigcup_{k=0}^{2^n - 1} \{(x,y,t)\in\mathbb{H}^1 \colon |(x,y)|\in A_{n,k}  \}\Big). \]

The set $P_{e}$ is porous with respect to the Euclidean distance, but not porous at the point $0$ with respect to the CC distance.
\end{proposition}

We prove Proposition \ref{enotcc} in Claim \ref{enotccclaim1} and Claim \ref{enotccclaim2}.

\begin{claim}\label{enotccclaim1}
The set $P_{e}$ is porous with respect to the Euclidean distance.
\end{claim}

\begin{proof}
We first verify that $P_{e}$ is porous with respect to the Euclidean distance at $0$. Since $P_{e}\subset \Lambda$, it suffices to show
\begin{equation}\label{eucat0}
B_{e}((0,0,1/n),1/3n)\cap \Lambda = \varnothing \mbox{ for every }n\in \mathbb{N}.
\end{equation}
Let $(x,y,t)\in B_{e}((0,0,1/n),1/3n)$. Then $|(x,y)|<1/3n$ and $|t|>2/3n$, in particular $|t|>|(x,y)|$. Hence $(x,y,t)\notin \Lambda$, proving \eqref{eucat0}.

Now we show Euclidean porosity at the remaining points of $P_{e}$: suppose that $(x,y,t) \in P_{e}\setminus \{0\}$. Then there exists $p\in \mathbb{N}$ and $0\leq k\leq 2^{p}-1$ such that $|(x,y)|\in A_{p,k}$. Since $A_{p,k}$ is $(1/3)$-porous in $[2^{-p} + k2^{-2p}, 2^{-p} + (k+1)2^{-2p}]$ at $|(x,y)|$ and the sets $A_{p,k}$ only meet at their endpoints, we may find $r_{n}>0$ with $r_{n}\to |(x,y)|$ such that:
\begin{equation}\label{porosityAnk}
B_{e}(r_{n}, |r_{n}-|(x,y)||/3)\cap (A_{m,l} \cup \{0\})=\varnothing
\end{equation}
for every $m\in \mathbb{N}$ and every $0\leq l\leq 2^{m}-1$.

Define $(x_{n},y_{n})=r_{n}(x,y)/|(x,y)|$. Notice
\begin{equation}\label{july91}
|(x_{n},y_{n})|=r_{n} \quad \mbox{and} \quad |(x_{n},y_{n})-(x,y)|=|r_{n}-|(x,y)||.
\end{equation}
Since $r_{n}\to |(x,y)|$, we have $(x_{n}, y_{n}, t)\to (x,y,t)$. We claim that for $n\in \mathbb{N}$:
\begin{equation}\label{P1porosity}
B_{e}((x_{n},y_{n},t), d_{e}((x_{n},y_{n},t),(x,y,t))/3)\cap P_{e}=\varnothing.
\end{equation}
For this, suppose:
\[(a,b,c)\in B_{e}((x_{n},y_{n},t), d_{e}((x_{n},y_{n},t),(x,y,t))/3).\]
Using \eqref{july91}, we have
\[d_{e}((x_{n},y_{n},t),(x,y,t))=|(x_{n},y_{n})-(x,y)|=|r_{n}-|(x,y)||.\]
Hence, using the above equality, the triangle inequality, and \eqref{july91} again,
\[ ||(a,b)|-r_{n}| \leq |(a,b)-(x_{n},y_{n})|\leq |r_{n}-|(x,y)||/3.\]
This implies
\[|(a,b)|\in B_{e}(r_{n}, |r_{n}-|(x,y)||/3).\]
Using \eqref{porosityAnk}, we deduce that $|(a,b)|\neq 0$ and $|(a,b)|\notin A_{m,l}$ for any choice of $m\in \mathbb{N}$ and $0\leq l\leq 2^{m}-1$. Hence $(a,b,c)\notin P_{e}$, verifying \eqref{P1porosity}. This shows that $P_{e}$ is Euclidean porous at $(x,y,t)$ and completes the proof.
\end{proof}

\begin{claim}\label{enotccclaim2}
The set $P_{e}$ is not porous with respect to the CC distance at $0$.
\end{claim}

\begin{proof}

Let $\lambda>0$. It suffices to show that if $(x,y,t) \in \mathbb{H}^{1}$ and $d_{k}(x,y,t)$ is sufficiently small, then:
\begin{equation}\label{nonCCporous}
B_{k}( (x, y ,t), \lambda d_{k}(x, y, t))\cap P_{e} \neq \varnothing.
\end{equation}

\vspace{0.3cm}

\textit{Case 1:} Suppose $(x,y,t)\notin \Lambda$, so $|(x,y)| < |t|$. For each $s>0$, define:
\[p_{s}:=\left(x+\frac{sx}{|(x,y)|}, y+\frac{sy}{|(x,y)|}, t-\sqrt{15}s^2\right).\]
Without loss of generality we assume $t>0$. In the case $t<0$, one should instead choose $t+\sqrt{15}s^2$ as the final coordinate in the definition of $p_{s}$.
We first notice that $d_{k}(p_{s},(x,y,t))=2s$. Therefore for every $0<s<\lambda\sqrt{t}/2$ we have $p_s\in B_k((x,y,t), \lambda d_{k}(x, y, t))$. Since
\[\left| \left( x+\frac{sx}{|(x,y)|}, y+\frac{sy}{|(x,y)|} \right) \right| = |(x,y)|+s,\]
we see $p_s\in \Lambda$ if and only if
\[
t-\sqrt{15}s^2\leq |(x,y)|+s
\]
so it suffices to choose $s$ such that
\[
t-\sqrt{15}s^2\leq s.
\]
If $d_k(x,y,t)$ is sufficiently small, which forces $t$ to be small, this holds for every $s \in (\lambda\sqrt{t}/8, \lambda\sqrt{t}/2)$.
We now show that $p_s\in P_{e}$ for some such $s$. This holds provided the interval
\[ \left(|(x,y)|+ \frac{\lambda\sqrt{t}}{8},\ |(x,y)|+ \frac{\lambda\sqrt{t}}{2} \right)\]
intersects some set $A_{n,k}$. By making $d_{k}(x,y,t)$ small, we may assume that $|(x,y)|+ \frac{\lambda\sqrt{t}}{8}$ lies in $(0,1)$. By Lemma \ref{intobs}, it then suffices to prove that 
\begin{align}\label{yy}
\left( |(x,y)|+\frac{\lambda \sqrt{t}}{2} \right) - \left(|(x,y)|+\frac{\lambda\sqrt{t}}{8}\right)-5\left(|(x,y)|+\frac{\lambda\sqrt{t}}{8}\right)^2\geq 0.
\end{align}
The factor 5 instead of 4 in \eqref{yy} takes into account that Lemma \ref{intobs} requires a closed interval $[\theta, \theta+4\theta^{2}]$ rather than an open one. To prove \eqref{yy}, let $\delta=3\lambda/320$ and assume $|(x,y)|<\delta$, which implies $|(x,y)|^2<\delta |(x,y)|$. Using also the inequality $(a+b)^2\leq 2a^{2}+2b^{2}$, $t, \lambda \in (0,1)$ and $|(x,y)| < |t|$, we estimate as follows:
\begin{align*}
&\left(|(x,y)|+\frac{\lambda\sqrt{t}}{2}\right) - \left(|(x,y)|+\frac{\lambda\sqrt{t}}{8}\right)-5\left(|(x,y)|+\frac{\lambda\sqrt{t}}{8}\right)^2\\
&\qquad > \frac{7\lambda\sqrt{t}}{32}-10\delta |(x,y)|\\
&\qquad > \frac{7\lambda\sqrt{t}}{32}-10\delta \sqrt{t}\\
&\qquad \geq 0.
\end{align*}
This verifies \eqref{yy} so $p_{s}\in P_{e}$ for some $s$, verifying \eqref{nonCCporous} and completing the proof in this case.

\vspace{0.3cm}

\textit{Case 2:} Suppose $(x,y,t)\in \Lambda$, so $|t|\leq |(x,y)|$. For each $s>0$, define
\begin{equation}\label{psform}
p_{s}= \left(x+\frac{sx}{|(x,y)|}, y+\frac{sy}{|(x,y)|}, t\right).
\end{equation}
It is easy to show that $d_{k}(p_{s},(x,y,t))=s$ so $p_{s}\in B_{k}( (x, y ,t), \lambda d_{k}(x, y, t))$ whenever $0<s<\lambda |(x,y)|$. Clearly also
\begin{equation}\label{psprojnorm}
 \left| \left( x+\frac{sx}{|(x,y)|}, y+\frac{sy}{|(x,y)|} \right) \right| = |(x,y)|+s.
\end{equation}
Since $s>0$ and $|t|\leq |(x,y)|$, we deduce that $p_{s}\in \Lambda$.

For sufficiently small $|(x,y)|$, the interval $(|(x,y)|, |(x,y)|+\lambda |(x,y)|)$ will contain a subinterval of the form $[\theta, \theta + 4\theta^{2}]$ for some $\theta \in (0,1)$, which by Lemma \ref{intobs} necessarily meets some set $A_{n,k}$. Using \eqref{psform} and \eqref{psprojnorm}, this yields \eqref{nonCCporous} and proves the claim in this case.
\end{proof}

\begin{proposition}\label{ccnote}
Define the cusp $\Upsilon$ by
\[\Upsilon:=\{(x,y,t)\in \mathbb{H}^1 \colon |t| \geq 2|(x,y)|^{2}\}.\]
Define
\[P_{c}:=\Upsilon \cap \Big(\{0\}\cup \bigcup_{n=1}^{\infty} \bigcup_{k=0}^{2^n - 1} \{(x,y,t)\in\mathbb{H}^1 \colon |(x,y)|\in A_{n,k}  \}\Big).\]

The set $P_{c}$ is porous with respect to the CC distance, but not porous at the point $0$ with respect to the Euclidean distance.
\end{proposition}

We prove Proposition \ref{ccnote} in Claim \ref{ccnoteclaim1} and Claim \ref{ccnoteclaim2}.

\begin{claim}\label{ccnoteclaim1}
The set $P_{c}$ is porous with respect to the CC distance.
\end{claim}

\begin{proof}
We first verify that $P_{c}$ is porous with respect to the Koranyi distance at $0$. Since $P_{c}\subset \Upsilon$, it suffices to prove that
\begin{equation}\label{july11eq}
B_{k}((1/n,0,0), 1/3n)\cap \Upsilon=\varnothing \quad \mbox{ for every }n\in\mathbb{N}.
\end{equation}
If $(x,y,t)\in B_{k}((1/n,0,0), 1/3n)$ then
\[\Big|(x-\frac{1}{n},y)\Big|\leq \frac{1}{3n}\quad \mbox{ and }\quad \Big|t+\frac{2y}{n}\Big|\leq \frac{1}{9n^2}.\]
In particular:
\[\frac{2}{3n} \leq x\leq \frac{4}{3n},\quad |y|\leq \frac{1}{3n}\quad \mbox{and}\quad \Big|t+\frac{2y}{n}\Big|\leq \frac{1}{9n^2}.\]
Therefore:
\begin{align*}
|t|\leq \Big|t+\frac{2y}{n}\Big|+\frac{2|y|}{n} \leq \frac{1}{9n^2}+\frac{2}{3n^2}\leq \frac{x^2}{4}+\frac{3x^2}{2}< 2 |(x,y)|^2.
\end{align*}
We conclude that $(x,y,t)\notin \Upsilon$, verifying \eqref{july11eq}.

We now show CC porosity at the remaining points of $P_{c}$. Suppose that $(x,y,t) \in P_{c}\setminus \{0\}$. Using the definition of $P_{c}$, there exists $p\in \mathbb{N}$ and $0\leq k\leq 2^{p}-1$ such that $|(x,y)|\in A_{p,k}$. Since $A_{p,k}$ is $(1/3)$-porous as a subset of the interval $[2^{-p} + k2^{-2p}, 2^{-p} + (k+1)2^{-2p}]$ and the sets $A_{p,k}$ only meet at their endpoints, we may find $r_{n}>0$ with $r_{n}\to |(x,y)|$ such that:
\begin{equation}\label{usefulpc}
B_{e}(r_{n}, |r_{n}-|(x,y)||/3)\cap (A_{m,l} \cup \{0\})=\varnothing
\end{equation}
for every $m\in \mathbb{N}$ and every $0\leq l\leq 2^{m}-1$.

Define $(x_{n},y_{n})=r_{n}(x,y)/|(x,y)|$. Notice that $|(x_{n},y_{n})|=r_{n}$ and
\[|(x_{n},y_{n})-(x,y)|=|r_{n}-|(x,y)||.\]
Clearly $(x_{n}, y_{n}, t)\to (x,y,t)$. We claim that for every $n\in \mathbb{N}$:
\begin{equation}\label{usefulpc2}
B_{k}((x_{n},y_{n},t), d_{k}((x_{n},y_{n},t),(x,y,t))/3)\cap P_{c}=\varnothing.
\end{equation}
For this, suppose
\[(a,b,c)\in B_{k}((x_{n},y_{n},t), d_{k}((x_{n},y_{n},t),(x,y,t))/3).\]
Using \eqref{Hgplaw}, we see:
\[d_{k}((x_{n},y_{n},t),(x,y,t))=|(x-x_{n}, y-y_{n})|=|r_n-|(x,y)||.\]
Hence
\begin{align*}
|(a,b)-(x_{n},y_{n})| &\leq d_{k}((x_{n},y_{n},t),(x,y,t))/3\\
&= |r_n-|(x,y)||/3.
\end{align*}
We deduce that
\[|(a,b)|\in B_{e}(r_{n}, |r_{n}-|(x,y)||/3).\]
Using \eqref{usefulpc}, we deduce that $|(a,b)|\neq 0$ and $|(a,b)|\notin A_{m,l}$ for any choice of $m\in \mathbb{N}$ and $0\leq l\leq 2^{m}-1$. Hence $(a,b,c)\notin P_{c}$, verifying \eqref{usefulpc2} which proves that $P_{c}$ is porous with respect to the Koranyi metric.
\end{proof}

\begin{claim}\label{ccnoteclaim2}
The set $P_{c}$ is not porous with respect to the Euclidean distance at the point $0$.
\end{claim}

\begin{proof}
Let $\lambda>0$. It suffices to show that for any $(x,y,t) \in \mathbb{H}^{1}$ with $d_{e}(x,y,t)$ sufficiently small:
\begin{equation}\label{nonEporous1}
B_{e}( (x, y ,t), \lambda d_{e}(x, y, t))\cap P_{c} \neq \varnothing.
\end{equation}

Without loss of generality assume $t\geq 0$. It will be clear from the proof that a similar argument works if $t<0$. For $0\leq s<|(x,y)|$, let
\begin{equation}\label{pcdefqs}
q_{s}:=\left( x-\frac{sx}{|(x,y)|}, y-\frac{sy}{|(x,y)|}, t+s \right).
\end{equation}
If $t<0$, one can instead choose $t-s$ as the final coordinate in the definition of $q_{s}$. Clearly $d_{e}(q_{s},(x,y,t))=s\sqrt{2}$. Hence $s< \lambda |(x,y)|/\sqrt{2}$ implies
\[q_{s}\in B_{e}( (x, y ,t), \lambda d_{e}(x, y, t)).\]
Next notice
\begin{equation}\label{usefula}
 \left| \left(x-\frac{sx}{|(x,y)|}, y-\frac{sy}{|(x,y)|}\right) \right|=|(x,y)|-s.
 \end{equation}
Using the definition of $\Upsilon$, \eqref{pcdefqs} and \eqref{usefula}, we see $q_{s}\in \Upsilon$ if and only if
\begin{equation}\label{tosatisfy1}
t+s\geq 2(|(x,y)|-s)^2.
\end{equation}
Since $t\geq 0$, \eqref{tosatisfy1} holds whenever
\begin{equation}\label{tosatisfy2}
s\geq 2(|(x,y)|-s)^2.
\end{equation}
Ensure $(x,y,t)$ is chosen with $d_{e}(x,y,t)$ small enough so that \eqref{tosatisfy2} holds if $s=\lambda |(x,y)|/2\sqrt{2}$. Then \eqref{tosatisfy2} holds, and hence $q_{s}\in \Upsilon$, whenever $s$ satisfies
\begin{equation}\label{srange}
\frac{\lambda |(x,y)|}{2\sqrt{2}} \leq s< \frac{\lambda |(x,y)|}{\sqrt{2}}.
\end{equation}

Finally we observe that if $d_{e}(x,y,t)$ is sufficiently small, then the interval
\[\left( |(x,y)|- \frac{\lambda |(x,y)|}{\sqrt{2}},\ |(x,y)| - \frac{\lambda |(x,y)|}{2\sqrt{2}} \right)\]
will contain a subinterval of the form $[\theta, \theta + 4\theta^{2}]$ for some $\theta \in (0,1)$, hence by Lemma \ref{intobs} intersect some set $A_{n,k}$. Using \eqref{pcdefqs}, \eqref{usefula} and the definition of $P_{c}$, this implies $q_{s}\in P_{c}$ for some $s$ satisfying \eqref{srange}. For such $s$, we have
\[q_{s}\in B_{e}( (x, y ,t), \lambda d_{e}(x, y, t))\cap P_{c},\]
which proves \eqref{nonEporous1}.
\end{proof}

\begin{remark}
We expect that in any Carnot group there exist sets which are porous in each distance (CC or Euclidean) but not even $\sigma$-porous in the other distance. Such constructions and their justifications may be complicated, since it is harder to show a set is not $\sigma$-porous than to show it is not porous: one may need to use Foran systems or similar techniques \cite[Lemma 4.3]{Zaj87}.
\end{remark}

\section{Non-differentiability on a $\sigma$-Porous Set}\label{porousnondifferentiability}

In this section we construct a Lipschitz function which is subdifferentiable at no point of a given $\sigma$-porous set in a Carnot group (Theorem \ref{nonsubdiff}). We first give basic properties of (Pansu) subdifferentiability (Definition \ref{subdiff}), which are simple adaptations of similar statements in Banach spaces \cite{HMZW97}.

\begin{proposition}\label{basicproperties}
The following statements hold:
\begin{enumerate}
\item $f\colon \bbG\to \bbR$ is Pansu differentiable at $a\in \bbG$ if and only if $f$ and $-f$ are both Pansu subdifferentiable at $a$.
\item Suppose $f$ and $g$ are Pansu subdifferentiable at $a$ and $\lambda>0$. Then $f+g$ and $\lambda f$ are Pansu subdifferentiable at $a$.
\item If 
\[\limsup_{h\to 0} \frac{f(ah)+f(ah^{-1})-2f(a)}{d_{c}(h)}>0\]
then $-f$ is not Pansu subdifferentiable at $a$.
\item $f\colon \bbR \to \bbR$ is subdifferentiable at $a$ if and only if $f_{+}(a)\geq f^{-}(a)$, where 
\[f_{+}(a)=\liminf_{t\to 0+} \frac{f(a+t)-f(a)}{t}\]
denotes the lower right Dini derivative, and 
\[f^{-}(a)=\limsup_{t\to 0-} \frac{f(a+t)-f(a)}{t}\]
denotes the upper left Dini derivative.
\item Suppose $L\colon \bbG\to \bbR$ is a non-zero group linear map. If $H\colon \bbR\to \bbR$ is not subdifferentiable at $L(a)$, then $H\circ L$ is not Pansu subdifferentiable at $a$. 
\item If $f\colon \bbG\to \bbR$ attains a local minimum at a point $x_0 \in\bbG$, then $f$ is Pansu subdifferentiable at $x_0$.
\end{enumerate}
\end{proposition}

\begin{proof} 
For (1), clearly Pansu differentiability of $f$ implies subdifferentiability of $f$ and $-f$. We check the opposite implication. Suppose $f$ and $-f$ are Pansu subdifferentiable. Then there exist group linear maps $L_{1}, L_{2}\colon \bbG \to \bbR$ such that
\begin{equation}\label{fsub} \liminf_{h\to 0} \frac{f(x_{0}h)-f(x_{0})-L_{1}(h)}{d_{c}(h)}\geq 0,\end{equation}
\begin{equation}\label{mfsub} \liminf_{h\to 0} \frac{f(x_{0})-f(x_{0}h)-L_{2}(h)}{d_{c}(h)}\geq 0.\end{equation}
Adding \eqref{fsub} and \eqref{mfsub} yields
\begin{equation}\label{used1} \liminf_{h\to 0} \frac{-L_{1}(h)-L_{2}(h)}{d_{c}(h)}\geq 0.\end{equation}
For each $v\in \bbG$, let $h=\delta_{t}(v)$. Group linearity of $L_{1}$ and $L_{2}$ implies
\begin{equation}\label{used2}\frac{-L_{1}(h)-L_{2}(h)}{d_{c}(h)}=\frac{-L_{1}(v)-L_{2}(v)}{d_{c}(v)}.\end{equation}
Letting $t\to 0$ and using \eqref{used1} and \eqref{used2} yields $-L_{1}(v)-L_{2}(v)\geq 0$. Hence $L_{1}(v)+L_{2}(v)\leq 0$ for every $v\in \bbG$. Replacing $v$ by $v^{-1}$ yields the opposite inequality $L_{1}(v)+L_{2}(v)\geq 0$ for every $v\in \bbG$. Hence $L_{2}=-L_{1}$. Pansu differentiability of $f$ follows directly from this equality, \eqref{fsub} and \eqref{mfsub}.

Statement (2) is trivial.

Suppose the condition in (3) holds but $-f$ is Pansu subdifferentiable at $a$ with corresponding map $L$. Since $L$ is group linear, $L(h^{-1})=-L(h)$. Hence
\begin{align*}
&\frac{-f(ah)+f(a)-L(h)}{d_{c}(h)}+\frac{-f(ah^{-1})+f(a)-L(h^{-1})}{d_{c}(h)}\\
&\qquad =-\frac{f(ah)+f(ah^{-1})-2f(a)}{d_{c}(h)}.
\end{align*}
Consequently either
\[ \liminf_{h\to 0} \frac{-f(ah)+f(a)-L(h)}{d_{c}(h)}<0\]
or
\[ \liminf_{h\to 0} \frac{-f(ah^{-1})+f(a)-L(h^{-1})}{d_{c}(h)}<0.\]
This contradicts Pansu subdifferentiability of $-f$ at $a$, proving (3).

Statement (4) is exactly as stated in \cite{HMZW97}.

We now prove (5). Since $H$ is not subdifferentiable at $L(a)$, we know by (4) that $H_{+}(L(a))<H^{-}(L(a))$. By \cite[Theorem 19.2.1]{BLU07}, every element of $\bbG$ is a product of elements of the form $\exp(X)$ with $X\in V_1$. Hence, since $L$ is a non-zero group linear map, there exists $X\in V_{1}$ such that $L(\exp(X)) \neq 0$. To show $H\circ L$ is not Pansu subdifferentiable at $a$, we show $f\colon \bbR\to \bbR$ given by $f(t)=H(L(a\exp(tX)))$ is not subdifferentiable at $0$. Notice:
\begin{align*}
\frac{f(t)-f(0)}{t}&=\frac{H(L(a\exp(tX)))-H(L(a))}{t}\\
&= \frac{H(L(a)+tL(\exp(X)))-H(L(a))}{tL(\exp(X))}\cdot L(\exp(X)).
\end{align*}
If $L(\exp(X))>0$ then
\[f_{+}(0)=H_{+}(L(a))L(\exp(X))\]
and
\[f^{-}(0)=H^{-}(L(a))L(\exp(X)).\]
Hence $f_{+}(0)<f^{-}(0)$, so $f$ is not subdifferentiable at $0$. If $L(\exp(X))<0$, then
\[f_{+}(0)=H^{-}(L(a))L(\exp(X))\]
and
\[f^{-}(0)=H_{+}(L(a))L(\exp(X))),\]
so again $f_{+}(0)<f^{-}(0)$. We conclude that $f$ is not subdifferentiable at $0$. Hence (5) holds. 

We now verify (6). If $f$ has a local minimum at $x_0$, then $f(x_0 h)\geq f(x_0)$ for all $h\in\bbG$ with $d_{c}(h)$ sufficiently small. Therefore:
\[\frac{f(x_0 h)-f(x_0)}{d_{c}(h)}\geq 0\quad \mbox{ for } h\in\bbG\setminus\{0\} \mbox{ with }d_{c}(h)\mbox{ sufficiently small},\]
and (6) follows.
\end{proof}

The following lemma \cite{HMZW97} will be used in Section \ref{gradientproblem}. There is no similar statement if $\bbR$ is replaced by $\bbR^n$ ($n>1$) or $\mathbb{H}^{n}$: these spaces admit measure zero sets containing points of (Pansu) differentiability for every real-valued Lipschitz function \cite{Pre90, PS16.1, LDPS17}.

\begin{lemma}\label{nullnonsubdiff}
Let $Z\subset \bbR$ have Lebesgue measure zero. Then there exists a Lipschitz function on $\bbR$ which is subdifferentiable at no point of $Z$.
\end{lemma}

In Banach spaces which admit a suitable bump function, one can construct a Lipschitz function which is differentiable at no point of any given $\sigma$-porous set. We show that the same is true in a general Carnot group. Our proof is a modification of \cite[Lemma 2]{HMZW97}. The next lemma is \cite[Lemma 3]{HMZW97}, where the set is assumed to be `uniformly porous'. In the present paper `uniformly porous' simply means `porous', since our definition of porosity already requires the relative size of holes (i.e. the parameter $\lambda$ in Definition \ref{def_porous}) to be uniform over the porous set. We denote the closed ball of center $x$ and radius $r>0$ in a metric space by $\overline{B}(x,r)$.

\begin{lemma}\label{cover}
Let $M$ be a metric space and $E$ be a closed porous subset of $M$. Then there exists $C>1$ and $S\subset M\times (0,1)$ such that the family $\mathcal{B}=\{ \overline{B}(x,r)\colon (x,r)\in S\}$ is disjoint, $\bigcup \mathcal{B} \cap E=\varnothing$, and for each $\delta>0$:
\[ \bigcup \mathcal{B} \cup \bigcup \{\overline{B}(x,Cr)\colon (x,r)\in S,\ r<\delta \}=M.\]
\end{lemma}

\begin{definition}
A \emph{bump function} $b\colon \bbG\to \bbR$ is a Lipschitz function which is everywhere Pansu differentiable, has compact support, is non-negative and is not identically zero.
\end{definition}

\begin{lemma}
There exists a bump function $b\colon \bbG \to \bbR$ with $\mathrm{Lip}(b)\leq 1$.
\end{lemma}

\begin{proof}
Let $b\colon \bbR^{n}\to \bbR$ be a non-negative $C^{1}$ function (in the classical sense) with compact support which is not identically zero. We can choose $b$ to have Lipschitz constant as small as we desire, without changing the support of $b$. Classically $C^{1}$ functions are Pansu differentiable with continuously varying derivative \cite[Remark 5.9]{FSSC2}, so $b$ is also Pansu differentiable. 

The CC distance is bounded below by a multiple of the Euclidean distance on compact sets \cite{Monti}, so $b$ is also Lipschitz with respect to the CC distance in the domain. Since the Euclidean Lipschitz constant could be made arbitrarily small, we can ensure the CC Lipschitz constant is at most 1. 
Finally, the CC distance and the Euclidean distance induce the same topology \cite{Monti}, so $b$ has compact support also with respect to the CC distance.
\end{proof}

\begin{theorem}\label{nonsubdiff}
Let $E$ be a $\sigma$-porous subset of $\bbG$. Then there is a Lipschitz function $f\colon \bbG\to \bbR$ which is Pansu subdifferentiable at no point of $E$. 
\end{theorem}

\begin{proof}
Let $b$ be a bump function on $\bbG$. By composing with group translations and dilations if necessary, we may assume that $b(0)=\beta>0$, $b$ is supported in $B(0,1)$ and $b$ is $1$-Lipschitz. Write $E=\cup_{i=1}^{\infty}E_{i}$ where each set $E_{i}$ is porous. We apply Lemma \ref{cover} to each set $E_{i}$ considered as a subset of the metric space $M_{i}=\bbG\setminus (\overline{E}_{i}\setminus E_{i})$. Then $E_{i}$ is porous and closed in $M_{i}$. Choose $S_{i}$ and $C_{i}>1$ corresponding to $E_{i} \subset M_{i}$ using Lemma \ref{cover}.

For each $i$, the family $\mathcal{B}_{i}^{\ast}=\{B(x,r)\colon (x,r)\in S_{i}\}$ is disjoint, where the balls are defined in $\bbG$ (not the subspace $M_{i}$). Indeed, if two members of $\mathcal{B}_{i}^{\ast}$ did intersect in $\bbG$ then, using Lemma \ref{cover} and the definition of $M_{i}$, this intersection must be in $\overline{E_{i}}$. Since the intersection is open, it would contain a point of $E_{i} \subset M$ which is impossible.

Lemma \ref{cover} implies that for every $\delta>0$,
\[\bigcup \{\overline{B}(x,C_{i}r)\colon (x,r)\in S_{i},\ r<\delta \}\supset E_{i}.\]
For each $i$, define $f_{i}\colon \bbG\to \bbR$ by:
\[f_{i}(x) = \begin{cases} 0 & \mbox{if } x\notin \bigcup \mathcal{B}_{i}^{\ast}, \\ rb(\delta_{1/r}(y^{-1}x)) & \mbox{if } x\in B(y,r),\ (y,r)\in S_{i}. \end{cases}\]
Since the CC distance is invariant under left translations, compatible with dilations and $b$ is $1$-Lipschitz, it follows $x\mapsto rb(\delta_{1/r}(y^{-1}x))$ is also $1$-Lipschitz. Each map $f_{i}$ is a supremum of $1$-Lipschitz functions, hence $1$-Lipschitz. Since also $S_{i}\subset \bbG\times (0,1)$, we have $0\leq f_{i}\leq 2$. The maps $f_{i}$ are Pansu differentiable on $\bigcup \mathcal{B}_{i}^{\ast}$ because $b$ is Pansu differentiable on $\bbG$. Clearly also $f_{i}(x)=r\beta$ for $(x,r)\in S_{i}$. 

Suppose $x\in E_{i}$. Then for arbitrarily small $r>0$, we find $(z,r)\in S_{i}$ such that $x\in \overline{B}(z,C_{i}r)$. Hence $d_{c}(x,z)\leq C_{i}r$ and $f_{i}(z)=r\beta$, which implies that $f_{i}(z)/d_{c}(x,z)\geq \beta/C_{i}$. We can write $z=xh$, where $h=x^{-1}z$ and $d_{c}(h)\leq C_{i}r$. Since $E_{i}\cap \bigcup \mathcal{B}_{i}^{\ast}=\varnothing$, we know $f_{i}(x)=0$. Letting $r\to 0$ gives:
\begin{equation}\label{july11a}\limsup_{h \to 0} \frac{f_{i}(xh)+f_{i}(xh^{-1})-2f_{i}(x)}{d_{c}(h)}\geq \frac{\beta}{C_{i}}>0.\end{equation}

If $x\in \bigcup \mathcal{B}_{i}^{\ast}$ then Pansu differentiability of $f_{i}$ at $x$ implies:
\begin{equation}\label{july11b}\lim_{h\to 0} \frac{f_{i}(xh)+f_{i}(xh^{-1})-2f_{i}(x)}{d_{c}(h)}=0.\end{equation}

If $x\notin \bigcup \mathcal{B}_{i}^{\ast}$ then $f_{i}(x)=0$ implies
\begin{equation}\label{july11c}\liminf_{h\to 0} \frac{f_{i}(xh)+f_{i}(xh^{-1})-2f_{i}(x)}{d_{c}(h)}\geq 0.\end{equation}

Define $f\colon \bbG\to \bbR$ by
\[f(x)=\sum_{i=1}^{\infty} \frac{f_{i}(x)}{2^{i}}.\]
Since each $f_{i}$ is $1$-Lipschitz, $f$ is $1$-Lipschitz. Let $j\in \bbN$ and $x\in E_{j}$. Then, for any $J>j$,
\begin{align*}
&\limsup_{h\to 0} \frac{f(xh)+f(xh^{-1})-2f(x)}{d_{c}(h)}\\
&\qquad \geq \limsup_{h\to 0} \frac{1}{2^j} \frac{f_{j}(xh)+f_{j}(xh^{-1})-2f_{j}(x)}{d_{c}(h)}\\
&\qquad \qquad + \liminf_{h\to 0} \sum_{\substack{1\leq i\leq J\\ i\neq j}} \frac{1}{2^i} \frac{f_{i}(xh)+f_{i}(xh^{-1})-2f_{i}(x)}{d_{c}(h)}\\
&\qquad \qquad -\sum_{i=J+1}^{\infty} \frac{2}{2^i}\\
&\qquad \geq \frac{\beta}{C_{j}}-\frac{1}{2^{J-1}}\\
&\qquad >0
\end{align*}
for $J$ sufficiently large, using \eqref{july11a}, \eqref{july11b}, \eqref{july11c} and that each $f_{i}$ is $1$-Lipschitz. Proposition \ref{basicproperties}(3) then asserts that $-f$ is not Pansu subdifferentiable on $E$, which proves the theorem.
\end{proof}

A universal differentiability set in $\bbG$ is a subset $A\subset \bbG$ such that every Lipschitz function $f\colon \bbG\to \bbR$ is Pansu differentiable at a point of $A$. Theorem \ref{nonsubdiff} gives the following corollary showing that, as in the Euclidean case, universal differentiability sets are far from being porous.

\begin{corollary}\label{udsnotporous}
A universal differentiability set in a Carnot group cannot be $\sigma$-porous.
\end{corollary}

\section{The Horizontal Gradient}\label{gradientproblem}

We next use Theorem \ref{nonsubdiff} and arguments from \cite{HMZW97} to prove Theorem \ref{gradienttheorem}. Roughly, this states that preimages of open sets under the horizontal gradient are far from being $\sigma$-porous. We first prove a useful lemma. Recall that $m$ is the dimension of the horizontal layer $V_{1}$ of $\bbG$.

\begin{lemma}\label{usefullemma}
Let $B(z,r)\subset \bbG$ be an open ball and $E\subset B(z,r)$. Fix parameters $v,r, \rho, \theta >0$ such that $8m\theta<\rho r v$. Let $b\colon \bbG \to \bbR$ be a 1-Lipschitz bump function satisfying $b(0)=v$ and supported in $B(0,1)$. Suppose there exists a continuous function $F \colon \overline{B}(z,r) \to \bbR$ such that:
\begin{itemize}
\item $|F(x)|\leq \theta$ for each  $x\in B(z,r)$.
\item The horizontal gradient $\nabla_{H} F(x)\in \bbR^{m}$ exists at each $x\in B(z,r)$, and $|\nabla_{H} F(x)|>\rho$ for each $x\in B(z,r)\setminus E$.
\end{itemize}
Then every Lipschitz function $h\colon \bbG\to \bbR$ is Pansu subdifferentiable at a point of $E$.
\end{lemma}

\begin{proof}
Let $\widetilde{b}(x):=b(\delta_{1/r}(z^{-1}x))$. Clearly $\widetilde{b}(z)=v$ and $\widetilde{b}$ is $(1/r)$-Lipschitz and supported in $B(z,r)$. Since $8m\theta<\rho r v$, we can choose $\eta$ satisfying $4\theta/v<\eta <\rho r/2m$. Define
\[G(x):=F(x)-\eta \widetilde{b}(x).\]
If $x\in \partial B(z,r)$ then $\widetilde{b}(x)=0$. Using also $|F(x)|\leq \theta$, we see
\begin{equation}\label{ppp}
|G(x)|=|F(x)-\eta \widetilde{b}(x)|\leq \theta \quad \mbox{ for }x\in \partial B(z,r).
\end{equation}
Moreover,
\begin{equation}\label{pppp}
G(z)=F(z)-\eta \widetilde{b}(z)\leq \theta-\eta v\leq -3\theta.
\end{equation}
Let $h\colon \bbG\to \bbR$ be a Lipschitz function. We intend to prove that $h$ is Pansu subdifferentiable at a point of $E$, for which we may assume that $h$ is not identically zero on $B(z,r)$. Define
\[a:=\sup_{\overline{B}(z,r)} |h|, \qquad c:=\min\left\{\frac{\theta}{2a},\frac{\rho}{4 m \mathrm{Lip}(h)}\right\}.\]
Then $\widetilde{h}(x):= c h(x)$ satisfies
\begin{align}\label{ineq1}
&|\widetilde{h}(x)|\leq \theta/2 \quad  \mbox{for } x\in \overline{B}(z,r),\\
\label{ineq2}
&\widetilde{h}\ \mbox{is}\  (\rho/4m)\mbox{-Lipschitz on}\  \overline{B}(z,r).
\end{align}
Let $ H(x):=G(x)+\widetilde{h}(x)$ for $x\in \overline{B}(z,r)$. By \eqref{ppp} and \eqref{ineq1}, we see
\begin{align}\label{p}
H(x)\geq -3\theta/2\quad \mbox{for } x\in\partial B(z,r).
\end{align}
Using \eqref{pppp} and \eqref{ineq1}, we get 
\begin{align}\label{q}
H(z)\leq -5\theta/2.
\end{align}
Putting together \eqref{p} and \eqref{q} we infer that $H$ attains its minimum at a point $x_0\in B(z,r)$. By Proposition \ref{basicproperties}(6), $H$ is Pansu subdifferentiable at $x_0$. Since $G$ is Pansu differentiable at $x_0$, $\widetilde{h}$ is Pansu subdifferentiable at $x_0$. 

We claim that $x_0\in E$. Suppose $x_0\notin E$. Then by the properties of $F$ in the statement of the lemma, we know $|\nabla_{H}F(x_0)|>\rho$. By the definition of the horizontal gradient, there is $1\leq i\leq m$ such that $|X_i F(x_0)|>\rho/m$. Without loss of generality we suppose $X_i F(x_0)<-\rho/m$: otherwise in what follows consider directional derivatives in direction $-X_{i}$ instead of in direction $X_{i}$. Since $\widetilde{b}$ is $(1/r)$-Lipschitz and $4\theta/v<\eta <\rho r/2m$, we have
\[|X_i(\eta \widetilde{b})(x_0)|\leq \mathrm{Lip}(\eta \widetilde{b})<\rho/2m.\]
Thus, since $G=F-\eta \widetilde{b}$,
\[X_iG(x_0)=X_iF(x_0)-X_i(\eta \widetilde{b})(x_0)<-\rho/2m.\]
To conclude, notice $H=G+\widetilde{h}$, $X_iG(x_0)<-\rho/2m$ and $\mathrm{Lip}(\widetilde{h})\leq \rho/4m$. This implies that $H$ does not attain its minimum at $x_0$, a contradiction. Hence $x_{0}\in E$, so $h$ is Pansu subdifferentiable at a point of $E$.
\end{proof}

\begin{theorem}\label{gradienttheorem}
Let $D\subset \bbG$ be an open set and $f\colon D\to \bbR$ be Pansu differentiable. Denote $g(x)=\nabla_{H} f(x)$ for $x\in D$ and suppose $G\subset \bbR^{m}$ is an open set such that $g^{-1}(G)\neq \varnothing$. Then the following statements hold:
\begin{enumerate}
\item $g^{-1}(G)$ is porous at none of its points.
\item If $T\subset \bbG$ is open and $T\cap g^{-1}(G)\neq \varnothing$, then $L(T\cap g^{-1}(G))$ has positive Lebesgue measure for every non-zero group linear $L~\colon~\bbG~\to~\bbR$. In particular, the one-dimensional Hausdorff measure of $T\cap g^{-1}(G)$ with respect to the CC metric is positive.
\item If $T\subset \bbG$ is open and $T\cap g^{-1}(G)\neq \varnothing$, then $T\cap g^{-1}(G)$ is not $\sigma$-porous.
\end{enumerate}
\end{theorem}

\begin{remark}\label{remark}
Properties (1)--(3) of Theorem \ref{gradienttheorem} hold if and only if the following is true:

\begin{enumerate}
\item[(4)] Suppose $a\in g^{-1}(G)$ and $B(z_{n},r_{n})$ is a sequence of open balls such that $z_{n}\to a$ and $r_{n}>cd_{c}(z_{n},a)$ for some $c>0$ and all $n$. Then there exists $n_{0}$ such that for all $n\geq n_{0}$:
\begin{enumerate}
\item[(a)] $L(g^{-1}(G)\cap B(z_{n},r_{n}))$ has positive Lebesgue measure for every non-zero group linear map $L\colon \bbG\to \bbR$.
\item[(b)] $g^{-1}(G)\cap B(z_{n},r_{n})$ is not $\sigma$-porous.
\end{enumerate}
\end{enumerate}
\end{remark}

\begin{proof}[Proof of Remark \ref{remark}]
We first assume (4) and prove (1)--(3). Suppose that $a\in g^{-1}(G)$, $z_{n}\to a$ and $r_{n}>cd(z_{n},a)$ for some fixed $c$ and every $n\in \bbN$. Then (4b) asserts that $g^{-1}(G)\cap B(z_n,r_n)$ is not $\sigma$-porous for all sufficiently large $n$, in particular it is non-empty. Hence $g^{-1}(G)$ cannot be porous at $a$. This proves (1).

Now let $T\subset\bbG$ be open and $T\cap g^{-1}(G)\neq \varnothing$. Choose $a\in T\cap g^{-1}(G)$, a sequence $z_{n}\to a$ with $z_{n}\neq a$ and let $r_{n}=d_{c}(z_n,a)/2$. Then (4) asserts that for sufficiently large $n$, $L(g^{-1}(G)\cap B(z_n,r_n))$ has positive Lebesgue measure for every group linear map $L$ and $g^{-1}(G)\cap B(z_n,r_n)$ is not $\sigma$-porous. Since $T$ is open, we have
\[ g^{-1}(G)\cap B(z_n,r_n) \subset g^{-1}(G)\cap T\]
for sufficiently large $n$. This yields (2) and (3). The statement about Hausdorff measure in (2) follows because group linear maps are Lipschitz, so the image of a set of one-dimensional Hausdorff measure zero would have one-dimensional Hausdorff measure zero.

We now assume (1)--(3) and prove (4). Suppose $a\in g^{-1}(G)$ and $B(z_{n},r_{n})$ is a sequence of open balls such that $z_{n}\to a$ and $r_{n}>cd_{c}(z_{n},a)$ for some $c>0$ and all $n$. By (1), $B(z_{n},r_{n})\cap g^{-1}(G)\neq \varnothing$ for all sufficiently large $n$. Properties (2) and (3) applied with $T=B(z_n,r_n)$ then give (4a) and (4b) respectively.
\end{proof}

\begin{proof}[Proof of Theorem \ref{gradienttheorem}]
Suppose $a, z_{n}, r_{n}, c$ are as in Remark \ref{remark}(4). Since making $r_{n}$ smaller makes the statement stronger, we may assume $r_{n}\to 0$. Choose a $1$-Lipschitz bump function $b\colon \bbG \to \bbR$ supported in $B(0,1)$ which satisfies $b(0)=v$ for some $v>0$. Since $G$ is open and $g(a)\in G$, there exists $\rho>0$ such that $|g(x)-g(a)|>\rho$ for any $x\not\in g^{-1}(G)$. Using Pansu differentiability of $f$ at $a$ and \eqref{gradform}, which expresses the Pansu derivative in terms of the horizontal gradient, we may find $\delta>0$ such that $d_{c}(x,a)<\delta$ implies
\[ |f(x)-f(a)-\langle g(a), p(a^{-1}x)\rangle |\leq \frac{\rho v d_{c}(x,a)}{16m(1+1/c)}.\]
Let
\[F(x):=f(x)-f(a)-\langle g(a), p(a^{-1}x)\rangle.\]
Choose $n_{0}$ such that $B(z_{n},r_{n})\subset B(a,\delta)$ for $n>n_{0}$. Fix $n>n_{0}$. For every $x\in B(z_{n},r_{n})$ we have, using $r_{n}>cd_{c}(z_{n},a)$ as in Remark \ref{remark}(4),
\begin{align*}
|F(x)| \leq \frac{\rho v d_{c}(x,a)}{16m(1+1/c)} &\leq \frac{\rho v(r_{n}+d_{c}(z_{n},a))}{16m(1+1/c)}\\
&< \frac{\rho v(r_{n}+r_{n}/c)}{16m(1+1/c)}\\
&=\frac{\rho vr_{n}}{16m}.
\end{align*}
For $x\notin g^{-1}(G)$ we have $|\nabla_{H} F(x)|=|g(x)-g(a)|>\rho$. Now the assumptions of Lemma \ref{usefullemma} hold with $z=z_{n}, r=r_{n}$, $\theta=\rho vr_{n}/16m$ and $E=B(z_{n},r_{n})\cap g^{-1}(G)$. Hence every real-valued Lipschitz map on $\bbG$ is Pansu subdifferentiable at a point of $E$.

To prove Remark \ref{remark}(4a), suppose $L(E)$ has measure zero for some non-zero group linear map $L\colon \bbG\to \bbR$. Using Lemma \ref{nullnonsubdiff} we can choose a Lipschitz function $H\colon \bbR \to \bbR$ which is subdifferentiable at no point of $L(E)$. Hence $h=H\circ L$ is a Lipschitz function which, by Proposition \ref{basicproperties}(5), is not Pansu subdifferentiable at any point of $E$. This contradicts Lemma \ref{usefullemma}.

To prove Remark \ref{remark}(4b), suppose $E$ is $\sigma$-porous. Then by Theorem \ref{nonsubdiff} there exists a Lipschitz function which is Pansu subdifferentiable at no point of $E$. This again contradicts Lemma \ref{usefullemma}.
\end{proof}


\begin{thebibliography}{99}
\bibitem{BLU07} {Bonfiglioli, A., Lanconelli, E., Uguzzoni, F.}: \emph{Stratified Lie groups and potential theory for their sub-Laplacians}, Springer Monographs in Mathematics 26, New York, Springer-Verlag (2007).
\bibitem{Buc} {Buczolich, Z.}: \emph{Solution of the gradient problem of C.E. Weil}, Rev. Math. Ibero. 21(3) (2005), 889--910.
\bibitem{CDPT07} {Capogna, L., Danielli, D., Pauls, S., Tyson, J.}: \emph{An introduction to the Heisenberg group and the sub-Riemannian isoperimetric problem}, Birkhauser, Progress in Mathematics, 259 (2007).
\bibitem{Cla} {Clarkson, J. A.}: \emph{A property of derivatives}, Bull. Amer. Math. Soc. 53 (1947), 125--126. 
\bibitem{CMPSC1} {Citti, G., Manfredini, M., Pinamonti, A., Serra Cassano, F.}: \emph{Smooth approximation for intrinsic Lipschitz functions in the Heisenberg group}, Calc. Var. Partial Differential Equations 49(3--4) (2014), 1279--1308.
\bibitem{CMPSC2} {Citti, G., Manfredini, M., Pinamonti, A., Serra Cassano, F.}: \emph{Poincar\'e-type inequality for Lipschitz continuous vector fields}, J. Math. Pures Appl. (9) 105(3) (2016), 265--292.
\bibitem{D} Denjoy, A.: \emph{Sur une propriet\'e des fonctions d\'erive\'es}, Enseignement Math. 18 (1916), 320--328.
\bibitem{DM11} Dor\'e, M., Maleva, O.: \emph{A compact null set containing a differentiability point of every Lipschitz function}, Math. Ann. 351(3) (2011), 633--663.
\bibitem{DM12} Dor\'e, M., Maleva, O.: \emph{A compact universal differentiability set with Hausdorff dimension one}, Israel J. Math. 191(2) (2012), 889--900.
\bibitem{DM14} Dymond, M., Maleva, O.: \emph{Differentiability inside sets with upper Minkowski dimension one}, Michigan Math. J., 65(3) (2016), 613--636.
\bibitem{FSC2} {Franchi, B., Serapioni, R., Serra Cassano, F.}: \emph{Differentiability of intrinsic Lipschitz functions within Heisenberg groups}, J. Geom. Anal. 21(4) (2011), 1044--1084.
\bibitem{FSSC2} {Franchi, B., Serapioni, R. and Serra Cassano, F.}: \emph{Rectifiability and perimeter in the Heisenberg Group}, Math Ann 321, (2001), 479--531.
\bibitem{Gro96} {Gromov, M.}: \emph{Carnot-Caratheodory spaces seen from within}, Progress in Mathematics, 144 (1996), 79--323.
\bibitem{Hei01} {Heinonen, J.}: \emph{Lectures on analysis on metric spaces}, Universitext. Springer-Verlag, New York (2001).
\bibitem{HMZW97} {Holicky, P., Maly, J., Weil, C. E., Zajicek, L.}: \emph{A note on the gradient problem}, Real Analysis Exchange 22(1), 1996--97, 225--235.
\bibitem{LDPS17} {Le Donne, E., Pinamonti, A., Speight, G.}: \emph{Universal differentiability sets and maximal directional derivatives in Carnot groups}, preprint available at https://arxiv.org/abs/1705.05871.
\bibitem{Mag01} {Magnani, V.}: \emph{Differentiability and area formula on stratified Lie groups}, Houston J. Math. 27(2) (2001), 297--323.
\bibitem{Mag} {Magnani, V.}: \emph{Towards differential calculus in stratified groups}, J. Aust. Math. Soc. 95(1) (2013), 76--128.
\bibitem{Mon02} {Montgomery, R.}: \emph{A tour of subriemannian geometries, their geodesics and applications}, American Mathematical Society, Mathematical Surveys and Monographs, 91 (2006).
\bibitem{Monti} {Monti, R.}: \emph{Distances, boundaries and surface measures in Carnot-Carath\'eodory spaces}, PhD Thesis Series Dipartimento di Matematica Universit\'a di Trento. Available at http://www.math.unipd.it/$\sim $monti/papers/TesiFinale.pdf
\bibitem{MS01} {Monti, R., Serra Cassano, F.}: \emph{Surface measures in Carnot-Caratheodory spaces}, Calc. Var. Partial Diff. Eq. 13 (2001), 339--376.
\bibitem{Pan89} {Pansu, P.}: \emph{Metriques de Carnot-Carath\'{e}odory et quasiisometries des espaces symetriques de rang un}, Annals of Mathematics 129(1) (1989), 1--60.
\bibitem{Pre90} Preiss, D.: \emph{Differentiability of Lipschitz functions on Banach spaces}, J. Funct. Anal. 91(2) (1990), 312--345.
\bibitem{PS15} {Preiss, D., Speight, G.}: \emph{Differentiability of Lipschitz functions in Lebesgue null sets}, Inventiones Mathematicae 199(2) (2015), 517--559.
\bibitem{PS16.1} {Pinamonti, A., Speight, G.}: \emph{A measure zero universal differentiability set in the Heisenberg group}, Mathematische Annalen 368(1) (2017), 233--278.
\bibitem{PS16.2} {Pinamonti, A., Speight, G.}: \emph{Porosity, differentiability and Pansu's theorem}, Journal of Geometric Analysis, 27(3) (2017), 2055--2080; 
\bibitem{SC16} {Serra Cassano, F.}: \emph{Some topics of geometric measure theory in Carnot groups}, Geometry, analysis and dynamics on sub-Riemannian manifolds 1, 1--121, EMS Series of Lectures in Mathematics.
\bibitem{SCV} {Serra Cassano, F., Vittone, D.}: \emph{Graphs of bounded variation, existence and local boundedness of non-parametric minimal surfaces in the Heisenberg group},  Adv. Calc. Var. 7(4) (2014), 409--492.
\bibitem{Sem96} {Semmes, S.}: \emph{On the nonexistence of bi-Lipschitz parameterizations and geometric problems about $A_{\infty}$-weights}, Revista Matematica Iberoamericana, 12(2) (1996), 337--410.
\bibitem{Vit14} {Vittone, D.}: \emph{The regularity problem for sub-Riemannian geodesics}, CRM Series, 17, Ed. Norm., Pisa (2014), 193--226. Notes available online at http://cvgmt.sns.it/paper/2416/.
\bibitem{Zaj76} {Zajicek, L.}: \emph{Sets of $\sigma$-porosity and sets of $\sigma $-porosity $(q)$}, Casopis pro pestovani matematiky 101(4) (1976), 350--359.
\bibitem{Zaj87} {Zajicek, L.}: \emph{Porosity and $\sigma$-porosity}, Real Analysis Exchange 13(2) (1987/1988), 314--350.
\bibitem{Zaj05} {Zajicek, L.}: \emph{On $\sigma$-porous sets in abstract spaces}, Abstract and Applied Analysis 2005(5), 509--534.
\end{thebibliography}
\end{document}